\documentclass{amsart}
\usepackage{amssymb, amsthm, amsmath}
\usepackage{enumerate}
\usepackage{hyperref}
\usepackage[all]{xy}

\numberwithin{equation}{section}
\newtheorem{theorem}[equation]{Theorem}
\newtheorem{lemma}[equation]{Lemma}
\newtheorem{proposition}[equation]{Proposition}
\newtheorem{corollary}[equation]{Corollary}

\theoremstyle{definition}

\theoremstyle{remark}
\newtheorem{remark}[equation]{Remark}

\newcommand{\ob}{\operatorname{ob}}
\newcommand{\mor}{\operatorname{mor}}
\newcommand{\im}{\operatorname{im}}
\newcommand{\coker}{\operatorname{coker}}

\begin{document}

\title{A Note on Homology over Functor Categories}
\author{Ged Corob Cook}
\address{Department of Mathematics\\
Royal Holloway, University of London\\
Egham\\
Surrey TW20 0EX\\
UK}
\email{Ged.CorobCook.2012@live.rhul.ac.uk}
\thanks{This work was done as part of the author's PhD at Royal Holloway, University of London, supervised by Brita Nucinkis.}

\begin{abstract}

It is known that, for $C$ an abelian category and $I$ small, the functor category $C^I$ is again abelian; thus we can do homology in such categories, and examine how it relates to homology in $C$ itself. However, there does not seem to be any good reference collecting these ideas together. This article seeks to fill the gap by showing that homology in $C^I$ behaves as one would expect.

\end{abstract}

\maketitle

\section*{Introduction}

Functor categories arise in many situations in the wild. By the Yoneda lemma, for example, every locally small category $I$ has a full, faithful functor into $Set^I$. Indeed, the categories of presheaves on topological spaces arise in this way. A more concrete example is given by the category of left $G$-sets, for a group $G$: this can be identified with $Set^G$.

As stated in the abstract, given an abelian category $C$, we want to be able to use homology in the functor category $C^I$, which is again abelian by \cite[Functor Categories 1.6.4]{Weibel}, and compare it to homology in $C$. Given the ubiquity of such categories, one might expect to find some research on the subject. Indeed, work has been done for specific categories $I$: the Bredon cohomology of a group, for example, works in the category of functors from the orbit category $\mathcal{O}_\mathfrak{F}G$ to the category of $R$-modules for some ring $R$, where $G$ is a group, $\mathfrak{F}$ a family of subgroups, and $\mathcal{O}_\mathfrak{F}G$ is the category of $G$-spaces of the form $G/H$, $H \in \mathfrak{F}$, with $G$-maps between them. A simpler example is ordinary cohomology over a ring $R$: thinking of $G$ as a category with one object, whose morphisms are left-multiplication by elements of $G$, the category of $R[G]$-modules consists of functors from $G$ to the category of $R$-modules. In addition it is worth mentioning that $R$-modules themselves can be thought of as functors from $R$ to $Ab$, the category of abelian groups, except that in this case $R$ and the functors must be enriched over $Ab$ -- in the language of modules, this ensures that scalar multiplication distributes over addition in $R$.

On the other hand, there are basic facts which we can establish even without knowing anything about $I$. It is surprising that this has not been done before. However, we have been unable to find it anywhere, and thus we hope to give a good reference for future applications in homology theory. In particular, such a reference is needed for the author's forthcoming work, \cite{Myself2}. 

So, in Section \ref{funcat}, we define functor categories and show some basic properties, in particular that functors $F\colon C \rightarrow D$ induce functors $F^I\colon C^I \rightarrow D^I$, before looking specifically at functor categories over abelian categories, and showing in Lemma \ref{les} that there is a nice way of characterising exact sequences in such categories. This section is foundational: a lot of it can be found in \cite[Section 2.1]{risingc}.

In Section \ref{homdfun} we apply this framework to the derived functors of additive functors between abelian categories. We show that the functors $C^I \rightarrow D^I$ induced by the derived functors of $F$ form a homological $\delta$-functor, and are naturally isomorphic to the derived functors of $F^I$. Moreover, we show that there is a Grothendieck spectral sequence of such derived functors.

Finally, in Section \ref{hombifun} we consider the situation of bifunctors $F^{I \times J}\colon C^I \times D^J \rightarrow E^{I \times J}$ where $C, D$ and $E$ are abelian categories. In the case where, for example, $C$ has enough projectives but $D$ does not, a little more care is needed to show the existence of long exact sequences in both variables.

\section{Functor Categories}
\label{funcat}

Given a category $C$ and a small category $I$, one can construct a category whose objects are the functors $I \rightarrow C$, and whose morphisms are the natural transformations between these functors. Such a category is called a functor category, and written $C^I$.

We can think of objects in $C^I$ as diagrams in $C$, that is, pairs
\[
(\{A^i \in \ob(C) : i \in \ob(I)\}, \{(\alpha^{ij}\colon A^i \rightarrow A^j) \in \mor(C) : (i \rightarrow j) \in \mor(I) \}).
\]
When it is clear, we may write $(\{A^i\}, \{\alpha^{ij}\})$ or just $\{A^i\}$ for this. Similarly, a morphism
\[
f\colon (\{A^i\},\{\alpha^{ij}\}) \rightarrow (\{B^i\},\{\beta^{ij}\})
\]
in $C^I$ consists of a set $\{f^i : i \in I\}$ of morphisms in $C$ such that for all $i,j \in I$ the square
\[
\xymatrix{A^i \ar[r]^{f^i} \ar[d]_{\alpha^{ij}} & B^i \ar[d]^{\beta^{ij}} \\
A^j \ar[r]^{f^j} & B^j}
\]
commutes. When it is clear, we may write $\{f^i : i \in I\}$, or just $\{f^i\}$, for $f$. We call the $A^i$s the \emph{components} of $\{A^i\}$ and the $f^i$s the \emph{components} of $f$.

Observe that we get an \emph{$i$th projection functor} $\pi^i\colon C^I \rightarrow C$ for each $i \in I$, which takes the $i$th component of objects and morphisms in $C^I$.

\begin{lemma}
\label{nattr'}
For each morphism $i \rightarrow j$ in $I$, we define $\gamma^{ij}\colon \pi^i \rightarrow \pi^j$ to be a collection of morphisms $\gamma^{ij}_A$ in $C$, for $A \in C^I$, where $\gamma^{ij}_A\colon \pi^i(A) \rightarrow \pi^j(A)$ is just the morphism $A^i \rightarrow A^j$ between the components of $A$. Then $\gamma^{ij}$ is a natural transformation.
\end{lemma}
\begin{proof}
Given a morphism $f\colon A \rightarrow B$ in $C^I$, we need to check the square
\[
\xymatrix{A^i \ar[r]^{f^i} \ar[d]_{\gamma^{ij}_A} & B^i \ar[d]^{\gamma^{ij}_B} \\
A^j \ar[r]^{f^j} & B^j}
\]
commutes. But this is just saying that $f$ is a morphism, so we are done.
\end{proof}

In this paper, $I$ will always be a small category.

Suppose now that $F\colon C \rightarrow D$ is a functor between categories $C$ and $D$. For a morphism
\[
f\colon (\{A^i\},\{\alpha^{ij}\}) \rightarrow (\{B^i\},\{\beta^{ij}\})
\]
in $C^I$, we have that
\[
\{F(f^i)\}\colon (\{F(A^i)\},\{F(\alpha^{ij})\}) \rightarrow (\{F(B^i)\},\{F(\beta^{ij})\})
\]
is a morphism in $D^I$ because the square
\[
\xymatrix{A^i \ar[r]^{f^i} \ar[d]_{\alpha^{ij}} & B^i \ar[d]^{\beta^{ij}} \\
A^j \ar[r]^{f^j} & B^j}
\]
commutes, so
\[
\xymatrix{F(A^i) \ar[r]^{F(f^i)} \ar[d]_{F(\alpha^{ij})} & F(B^i) \ar[d]^{F(\beta^{ij})} \\
F(A^j) \ar[r]^{F(f^j)} & F(B^j)}
\]
does too. It is clear from the definition that composition of morphisms is preserved by this. Thus we get the following results.

\begin{proposition}
Define $F^I\colon C^I \rightarrow D^I$ by the maps
\[
(\{A^i\},\{\alpha^{ij}\}) \mapsto (\{F(A^i)\},\{F(\alpha^{ij})\})
\]
and
\begin{align*}
((\{A^i\},\{\alpha^{ij}\}) &\rightarrow (\{B^i\},\{\beta^{ij}\})) \mapsto \\
((\{F(A^i)\},\{F(\alpha^{ij})\}) &\rightarrow (\{F(B^i)\},\{F(\beta^{ij})\})).
\end{align*}
Then $F^I$ is a functor, which we call the \emph{exponent} of $F$ by $I$.
\end{proposition}

\begin{lemma}
\label{nattr}
Given functors $F,G\colon C \rightarrow D$ and a natural transformation $\eta\colon F \rightarrow G$, we get a natural transformation
\[
\eta^I\colon F^I \rightarrow G^I,
\]
where, for each $A \in C^I$,
\[
\eta^I_A\colon F^I(A) \rightarrow G^I(A)
\]
is the map with $i$th component
\[
\eta_{A^i}\colon F(A^i) \rightarrow G(A^i).
\]
\end{lemma}
\begin{proof}
To show that each $\eta^I_A$ is a morphism in $D^I$, we need to check that the squares
\[
\xymatrix{F(A^i) \ar[r]^{\eta_{A^i}} \ar[d]_{F(\alpha^{ij})} & G(A^i) \ar[d]^{G(\alpha^{ij})} \\
F(A^j) \ar[r]^{\eta_{A^j}} & G(A^j)}
\]
commute, which holds because $\eta$ is a natural transformation. To show $\eta^I$ is a natural transformation, it remains to show that, for a morphism $f\colon A \rightarrow B$ in $C^I$, the squares
\[
\xymatrix{F^I(A) \ar[r]^{F^I(f)} \ar[d]_{\eta^I_A} & F^I(B) \ar[d]^{\eta^I_B} \\
G^I(A) \ar[r]^{G^I(f)} & G^I(B)}
\]
commute; it suffices to show that each component commutes, which is just another application of the naturality of $\eta$.
\end{proof}

Given a categories $C,D$, and a functor $F\colon C \rightarrow D^I$, we will frequently write $F^i$ for the composite $\pi^i F$.

\begin{lemma}
\label{nattr''}
Given functors $F,G\colon C \rightarrow D^I$ and a natural transformation $\eta\colon F \rightarrow G$, we get a natural transformation
\[
\eta^i\colon F^i \rightarrow G^i,
\]
where, for each $A \in C$,
\[
\eta^i_A\colon F^i(A) \rightarrow G^i(A)
\]
is the $i$th component of
\[
\eta_A\colon F(A) \rightarrow G(A).
\]
\end{lemma}
\begin{proof}
Similar to the previous lemmas.
\end{proof}

A \emph{preadditive} category is a category enriched over the category $Ab$ of abelian groups; that is, one in which every set of morphism $\mor(A,B)$ is an abelian group, such that composition of morphisms is distributive over addition. In other words, a morphism $B \rightarrow B'$ induces a group homomorphism $\mor(A,B) \rightarrow \mor(A,B')$, and a morphism $A' \rightarrow A$ induces a group homomorphism $\mor(A,B) \rightarrow \mor(A',B)$. Over preadditive categories we can define additive functors: functors $F$ enriched over $Ab$, so that $\mor(A,B) \rightarrow \mor(F(A),F(B))$ is a group homomorphism.

An \emph{additive} category is a preadditive category that has a zero object (that is, an object that is both terminal and initial in its category) and pairwise products. It can be shown that this implies the pairwise products are also pairwise coproducts, so we call these biproducts. Over additive categories we can define kernels and cokernels of morphisms; see \cite[Appendix A]{Weibel} again.

An \emph{abelian} category is an additive category such that every morphism has a kernel and a cokernel, every monomorphism is the kernel of its cokernel, and every epimorphism is the cokernel of its kernel. Examples of abelian categories include the category of abelian groups, and the category of $R$-modules for a ring $R$. In abelian categories we can talk about chain complexes, i.e. sequences
\[
\cdots \rightarrow A_2 \xrightarrow{f_2} A_1 \xrightarrow{f_1} A_0 \rightarrow \cdots
\]
such that $f_{i-1} f_i = 0$ for each $i$, and exact sequences, i.e. chain complexes such that $\im (f_i) = \ker (f_{i-1})$ for each $i$.

\begin{lemma}
If $C$ is a preadditive category, so is $C^I$.
\end{lemma}
\begin{proof}
Suppose
\[
f,g\colon (\{A^i\},\{\alpha^{ij}\}) \rightarrow (\{B^i\},\{\beta^{ij}\})
\]
are two morphisms in $C^I$. We define
\[
f+g\colon (\{A^i\},\{\alpha^{ij}\}) \rightarrow (\{B^i\},\{\beta^{ij}\})
\]
by $(f+g)^i = f^i + g^i$. For this to be a morphism, we need the squares
\[
\xymatrix{A^i \ar[r]^{(f+g)^i} \ar[d]_{\alpha^{ij}} & B^i \ar[d]^{\beta^{ij}} \\
A^j \ar[r]^{(f+g)^j} & B^j}
\]
to commute. Now
\[
\beta^{ij} (f+g)^i = \beta^{ij} f^i + \beta^{ij} g^i = f^j \alpha^{ij} + g^j \alpha^{ij} = (f+g)^j \alpha^{ij},
\]
as required.

Write $e^i$ for the homomorphism $A^i \rightarrow B^i$ which is the identity in $Hom_C(A^i,B^i)$. Consider the square
\[
\xymatrix{A^i \ar[r]^{e^i} \ar[d]_{\alpha^{ij}} & B^i \ar[d]^{\beta^{ij}} \\
A^j \ar[r]^{e^j} & B^j \rlap{.}}
\]
Since composition distributes over addition in $C$, it follows that $e^j \alpha^{ij}$ and $\beta^{ij} e^i$ are both the identity of $Hom_C(A^i,B^j)$, so this square commutes. Hence $e=\{e^i\}$ is a morphism in $C^I$, and for any other morphism $f\colon\{A^i\}\rightarrow\{B^i\}$ in $C^I$, $e+f = \{e^i+f^i\} = \{f^i\} = f$, and similarly $f+e=f$, so $e$ is an identity element in $Hom_{C^I}(\{A^i\},\{B^i\})$.

One can show the existence of inverses similarly. Thus $C^I$ is preadditive.
\end{proof}

\begin{lemma}
\label{fun}
Suppose $C$ and $D$ are preadditive categories and $F\colon C \rightarrow D$ is additive. Then $F^I$ is additive.
\end{lemma}
\begin{proof}
Let $f,g$ be morphisms $\{A^i\} \rightarrow \{B^i\}$ in $C^I$. Then
\[
F^I(f+g) = \{F(f^i + g^i)\} = \{F(f^i)\}+\{F(g^i)\} = F^I(f)+F^I(g).
\]
Similarly for the other conditions.
\end{proof}

\begin{lemma}
\label{fun'}
Suppose $C$ is a preadditive category. Then $\pi^i\colon C^I \rightarrow C$ is additive.
\end{lemma}
\begin{proof}
$\pi^i(f+g) = (f+g)^i = f^i+g^i = \pi^i(f)+\pi^i(g)$.
\end{proof}

From now on, we will assume that our categories $C$ and $D$ are abelian (see \cite[Appendix A.4]{Weibel} for definitions) -- note that abelian categories are \textit{a fortiori} preadditive, so the previous results apply. It is known that $C^I$ is abelian (e.g. see \cite[Functor Categories 1.6.4]{Weibel}). We want to show that exact sequences in $C^I$ are just sequences in $C^I$ which are exact at each component. To show this, we need a preliminary lemma.

\begin{lemma}
Suppose $A=(\{A^i\},\{\alpha^{ij}\}), B=(\{B^i\},\{\beta^{ij}\})$, and consider $f\colon A \rightarrow B$ in $C^I$.
\begin{enumerate}[(i)]
\item The kernel $\ker(f)$ is the object $(\{\ker(f^i)\},\{\gamma^{ij}\})$ together with the morphism $g\colon \ker(f) \rightarrow A$, where $g^i$ is the canonical map $\ker(f^i) \rightarrow A^i$ in $C$, and $\gamma^{ij}$ is the (unique) morphism $\ker(f^i) \rightarrow \ker(f^j)$ given by the universal property of $\ker(f^j)$ in the diagram
\[
\xymatrix{\ker(f^i) \ar[r]^{f^i} \ar@{-->}[d]_{\gamma^{ij}} & A^i \ar[d]^{\alpha^{ij}} \\
\ker(f^j) \ar[r]^{f^j} & A^j \rlap{.}}
\]
\item Similarly for $\coker(f)$.
\end{enumerate}
\end{lemma}
\begin{proof}
We will prove (i), and leave it to the reader to check that $\ker(f)$ really is an element of $C^I$, that $g$ really is a morphism, and that (ii) goes through in the same way.

It is clear that
\[
fg\colon \ker(f) \rightarrow B
\]
is the zero map, since (one may check) the zero element $0^I$ of $C^I$ is the element with all its components the zero element $0$ of $C$, with identity morphisms between them. Suppose we have a morphism
\[
h\colon E=(\{E^i\},\{\varepsilon^{ij}\}) \rightarrow A
\]
such that $fh=0$. By definition, to show that $(\{\ker(f^i)\},\{\gamma^{ij}\})$ is the kernel of $f$, we need to show that there is a unique
\[
k\colon E \rightarrow (\{\ker(f^i)\},\{\gamma^{ij}\})
\]
such that $h = gk$. Now for each $i \in I$, $f^i h^i = 0$ in $C$, so again by definition of the kernel there is some unique
\[
k^i\colon E^i \rightarrow \ker(f^i)
\]
such that $h^i = g^i k^i$. To show that $h$ factors through $k=\{k^i\}$, we just need to check that the squares
\[
\xymatrix{E^i \ar[r]^{k^i} \ar[d]_{\varepsilon^{ij}} & \ker(f^i) \ar[d]^{\gamma^{ij}} \\
E^j \ar[r]^{k^j} & \ker(f^j)}
\]
commute. Then uniqueness follows from uniqueness of the $k^i$.

To see this, note that
\[
g^j k^j \varepsilon^{ij} = h^j \varepsilon^{ij} = \alpha^{ij} h^i = \alpha^{ij} g^i k^i = g^j \gamma^{ij} k^i,
\]
so, as $g^j$ is monic, it follows that $k^j \varepsilon^{ij} = \gamma^{ij} k^i$.
\end{proof}

Now it is also known that, for a morphism $f$ in an abelian category, $\im(f) = \ker \circ \coker(f)$: see \cite[p.~425]{Weibel}. Therefore:

\begin{lemma}
\label{les}
Given a sequence $L \xrightarrow{f} M \xrightarrow{g} N$ in $C^I$ such that $gf=0$,
\begin{tabbing}
it is exact at $M$ \= iff the canonical map $\im(f) \rightarrow \ker(g)$ is an isomorphism \\
\> iff the canonical map $\im(f^i) \rightarrow \ker(g^i)$ is an isomorphism for all $i$ \\
\> iff the sequence $L^i \xrightarrow{f^i} M^i \xrightarrow{g^i} N^i$ is exact at $M^i$ for all $i$.
\end{tabbing}
\end{lemma}

\section{Homological \texorpdfstring{$\delta$}{delta}-functors}
\label{homdfun}

In view of the fact that, when $C$ is abelian, $C^I$ is, it makes sense to compare homological properties over the two. We will assume, for simplicity, that all functors are covariant; dual statements follow by duality.

The following definition is taken from \cite[Definition 2.1.1]{Weibel}. We say $F$ is a \emph{homological $\delta$-functor} $C \rightarrow D$ if we have a collection of additive functors $F_n\colon C \rightarrow D$ for $n \in \mathbb{Z}$ such that, for each short exact sequence $0 \rightarrow L \rightarrow M \rightarrow N \rightarrow 0$ in $C$, we have a morphism $\delta_n\colon F_n(N) \rightarrow F_{n-1}(L)$, satisfying the following conditions.
\begin{enumerate}[(i)]
\item The functors $F_n$ are $0$ for $n<0$.
\item For each short exact sequence as above, there is a long exact sequence
\[
\cdots F_{n+1}(N) \xrightarrow{\delta_{n+1}} F_n(L) \rightarrow F_n(M) \rightarrow F_n(N) \xrightarrow{\delta_n} F_{n-1}(L) \rightarrow \cdots.
\]
\item For a morphism of short exact sequences
\[
\xymatrix{0 \ar[r] & L \ar[r] \ar[d]_f & M \ar[r] \ar[d] & N \ar[r] \ar[d]_g & 0 \\
0 \ar[r] & L' \ar[r] & M' \ar[r] & N' \ar[r] & 0 \rlap{,}}
\]
we get commutative squares
\[
\xymatrix{F_n(N) \ar[r]^{\delta_n} \ar[d]_{F_n(g)} & F_{n-1}(L) \ar[d]^{F_{n-1}(f)} \\
F_n(N') \ar[r]^{\delta_n} & F_{n-1}(L') \rlap{.}}
\]
\end{enumerate}

A morphism of homological $\delta$-functors $F \rightarrow G$ is a collection of natural transformations $F_n \rightarrow G_n$ that commute with the $\delta_n$.

Suppose $F$ is a homological $\delta$-functor $C \rightarrow D$. Then, as in Lemma \ref{fun}, the exponent functor $F_n^I\colon C^I \rightarrow D^I$ is additive for each $n$. Given a short exact sequence $0 \rightarrow L \rightarrow M \rightarrow N \rightarrow 0$ in $C^I$, for each $i \in I$ we have a short exact sequence $0 \rightarrow L^i \rightarrow M^i \rightarrow N^i \rightarrow 0$ in $C$, and hence a map $\delta_n\colon F_n(N^i) \rightarrow F_{n-1}(L^i)$. To show that we have a map $\delta_n^I\colon F_n^I(N) \rightarrow F_{n-1}^I(L)$ whose components are the $\delta_n$, we need the commutativity of
\[
\xymatrix{F_n(N^i) \ar[r]^{\delta_n} \ar[d] & F_{n-1}(L^i) \ar[d] \\
F_n(N^j) \ar[r]^{\delta_n} & F_{n-1}(L^j) \rlap{.}}
\]
This holds by part (iii) of the definition of a $\delta$-functor.

\begin{proposition}
\label{F^I}
The functors $F_n^I$ together with the maps $\delta_n^I$ form a homological $\delta$-functor from $C^I$ to $D^I$, which we denote $F^I$. Moreover, given homological $\delta$-functors $F,G\colon C \rightarrow D$ and a morphism $\eta\colon F \rightarrow G$, we get a morphism $\eta^I\colon F^I \rightarrow G^I$.
\end{proposition}
\begin{proof}
It is clear that $F_n^I=0$ for $n<0$. For each short exact sequence $0 \rightarrow L \rightarrow M \rightarrow N \rightarrow 0$, the exactness of the $i$th component at each $i$ makes
\[
\cdots \rightarrow F_{n+1}^I(N) \xrightarrow{\delta_{n+1}^I} F_n^I(L) \rightarrow F_n^I(M) \rightarrow F_n^I(N) \xrightarrow{\delta_n^I} F_{n-1}^I(L) \rightarrow \cdots
\]
exact, by the results of Lemma \ref{les}.

Finally, suppose we have a morphism of short exact sequences from $0 \rightarrow L \rightarrow M \rightarrow N \rightarrow 0$ to $0 \rightarrow L' \rightarrow M' \rightarrow N' \rightarrow 0$. To show the commutativity of
\[
\xymatrix{F_n^I(N) \ar[r]^{\delta_n^I} \ar[d] & F_{n-1}^I(L) \ar[d] \\
F_n^I(N') \ar[r]^{\delta_n^I} & F_{n-1}^I(L') \rlap{,}}
\]
we just need the commutativity of
\[
\xymatrix{F_n(N^i) \ar[r]^{\delta_n} \ar[d] & F_{n-1}(L^i) \ar[d] \\
F_n(N'^i) \ar[r]^{\delta_n} & F_{n-1}(L'^i)}
\]
for each $i$: this holds by part (iii) of the definition of a $\delta$-functor.

By Lemma \ref{nattr}, the natural transformations $\eta_n$ give rise to natural transformations $\eta_n^I$, so we just need to check the commutativity of
\[
\xymatrix{F_n^I(N) \ar[r]^{\delta_n^I} \ar[d]_{\eta_{n,N}^I} & F_{n-1}^I(L) \ar[d]^{\eta_{n-1,L}^I} \\
G_n^I(N) \ar[r]^{\varepsilon_n^I} & G_{n-1}^I(L) \rlap{,}}
\]
which holds because each of its components commutes.
\end{proof}

\begin{proposition}
If $F=\{F_n\}$ is a homological $\delta$-functor $C \rightarrow D^I$, then so is $F^i=\{F_n^i\}\colon C \rightarrow D$ with the differentials $\delta_n^i$, where, for a short exact sequence
\[
0 \rightarrow L \rightarrow M \rightarrow N \rightarrow 0,
\]
$\delta_n^i\colon F_n^i(N) \rightarrow F_{n-1}^i(L)$ is just the $i$th component of $\delta_n\colon F_n(N) \rightarrow F_{n-1}(L)$. Moreover, for each morphism $i \rightarrow j$ in $I$, we get a morphism of $\delta$-functors $F^i \rightarrow F^j$.
\end{proposition}
\begin{proof}
By Lemma \ref{fun'}, each $F_n^i$ is an additive functor. By Lemma \ref{les}, taking the $i$th component of a long exact sequence
\[
\cdots F_{n+1}(N) \xrightarrow{\delta_{n+1}} F_n(L) \rightarrow F_n(M) \rightarrow F_n(N) \xrightarrow{\delta_n} F_{n-1}(L) \rightarrow \cdots
\]
in $D^I$ gives a long exact sequence
\[
\cdots F_{n+1}^i(N) \xrightarrow{\delta_{n+1}^i} F_n^i(L) \rightarrow F_n^i(M) \rightarrow F_n^i(N) \xrightarrow{\delta_n^i} F_{n-1}^i(L) \rightarrow \cdots
\]
in $D$. Also, taking the $i$th component of any commutative square
\[
\xymatrix{F_n(N) \ar[r]^{\delta_n} \ar[d] & F_{n-1}(L) \ar[d] \\
F_n(N') \ar[r]^{\delta_n} & F_{n-1}(L')}
\]
in $D^I$ gives a commutative square
\[
\xymatrix{F_n^i(N) \ar[r]^{\delta_n^i} \ar[d] & F_{n-1}^i(L) \ar[d] \\
F_n^i(N') \ar[r]^{\delta_n^i} & F_{n-1}^i(L')}
\]
in $D$. Hence each $G^i$ is a homological $\delta$-functor.

By Lemma \ref{nattr'}, we get natural transformations $\gamma_n^{ij}\colon F_n^i \rightarrow F_n^j$ for each $n$; for a morphism of $\delta$-functors, we need to show that the squares
\[
\xymatrix{F_n^i(N) \ar[r]^{\delta_n^i} \ar[d]_{\gamma_n^{ij}} & F_{n-1}^i(L) \ar[d]^{\gamma_{n-1}^{ij}} \\
F_n^j(N) \ar[r]^{\delta_n^j} & F_{n-1}^j(L)}
\]
commute. This holds because $F$ is a homological $\delta$-functor.
\end{proof}

\emph{Projective} objects $P$ in an abelian category $C$ are defined by the following universal property: for any epimorphism $f\colon M \rightarrow N$, and any morphism $g\colon P \rightarrow N$, there is a morphism $h\colon P \rightarrow M$ such that $g=fh$. We say $C$ has \emph{enough projectives} if for every object $A$ in $C$ there is an epimorphism $P \rightarrow A$ for some projective $P$. In that case, we can take a projective resolution of $A$: namely, a sequence $\cdots \rightarrow P_1 \rightarrow P_0 \rightarrow 0$ with every $P_n$ projective such that $\cdots \rightarrow P_1 \rightarrow P_0 \rightarrow A \rightarrow 0$ is exact.

In an abelian category $C$ with enough projectives, if we are given an additive functor $F$, we can define the \emph{left derived functors} of $F$, $L_nF$, in the following way: for $A \in C$, take a projective resolution $\cdots \rightarrow P_1 \rightarrow P_0 \rightarrow 0$ of $A$, and then $L_nF(A)$ is the $n$th homology group of the chain complex $\cdots \rightarrow F(P_1) \rightarrow F(P_0) \rightarrow 0$.

It is well known that each $L_nF(A)$ is well defined, and that the $L_nF$ form a homological $\delta$-functor: see \cite[Lemma 2.4.1, Theorem 2.4.6]{Weibel}.

A homological $\delta$-functor $F$ is called \emph{universal} if, given another homological $\delta$-functor $G$ and a natural transformation $t\colon G_0 \rightarrow F_0$, there is a unique morphism $T=\{t_n\}\colon G \rightarrow F$ such that $t_0 = t$ (\cite[Definition 2.1.4]{Weibel}). Suppose $C$ has enough projectives: then given an additive functor $F\colon C \rightarrow D$, the $L_nF$ are universal by \cite[Theorem 2.4.7]{Weibel}.

If the additive functor $F$ is right-exact, one can show that $F=L_0F$. When $F$ is not right-exact, it follows by universality that $L_nF$ is naturally isomorphic to $L_n(L_0F)$ for each $n$, so that we do not gain anything by considering the more general situation.

\begin{lemma}
Suppose we are given a right-exact additive functor $F\colon C \rightarrow D$, and $C$ and $C^I$ have enough projectives. Then there is a natural isomorphism $(L_nF)^I \rightarrow L_n(F^I)$, for each $n$, of functors $C^I \rightarrow D^I$, which gives an isomorphism of $\delta$-functors.
\end{lemma}
\begin{proof}
Note that, since $F$ is right-exact, by Lemma \ref{les} $F^I$ is right-exact. So $(L_0F)^I = L_0(F^I) = F^I$ and we get the identity $(L_0F)^I \rightarrow L_0(F^I)$. Now apply the universal property of left derived functors to extend this to a morphism of $\delta$-functors $(L_nF)^I \rightarrow L_n(F^I)$. Finally, note that each component of a projective in $C^I$ must be projective in $C$, since by Lemma \ref{les} each component of an epimorphism in $C^I$ must be an epimorphism in $C$. So by Lemma \ref{les} again a projective resolution of an object $A$ in $C^I$ gives projective resolutions in $C$ to each of its components $A^i$, and hence $(L_nF)^i(A)=L_nF(A^i)$, so $(L_nF)^I(A)=L_n(F^I)(A)$, as required.
\end{proof}

As a result of this, we will just write $L_nF^I$ for $(L_nF)^I$ in the case that $C$ has enough projectives, whether $C^I$ does or not.

We now give a standard result of homology: the Grothendieck spectral sequence. For the proof, see \cite[Theorem 5.8.3, Theorem 5.5.1]{Weibel}. Given a right-exact functor $F\colon C \rightarrow D$, where $C$ has enough projectives, an object $A$ of $C$ is called $F$-acyclic if $L_nF(A)=0$ for all $n>0$.

\begin{theorem}
\label{gss}
Suppose that $C$, $D$ and $E$ are abelian categories, and that $C$ and $D$ have enough projectives. Suppose we have right-exact functors $F\colon C \rightarrow D$ and $G\colon D \rightarrow E$ such that $F$ sends projective objects of $C$ to $G$-acyclic objects of $D$. Then there is a convergent first quadrant homology spectral sequence for each $A \in C$:
\[
E_{pq}^2 = (L_pG)(L_qF)(A) \Rightarrow L_{p+q}(GF)(A).
\]
Moreover the convergence is natural in the sense that, given a morphism $A \rightarrow B$, the induced map
\[
L_{p+q}(GF)(A) \rightarrow L_{p+q}(GF)(B)
\]
is compatible with the induced map of spectral sequences
\[
(L_pG)(L_qF)(A) \rightarrow (L_pG)(L_qF)(B).
\]
\end{theorem}

\begin{corollary}
\label{gss'}
For $C,D,E,F,G$ as before, $I$ a small category and $A \in C^I$, there is a convergent first quadrant homology spectral sequence:
\[
E_{pq}^2(A) = (L_pG^I)(L_qF^I)(A) \Rightarrow L_{p+q}(GF)^I(A).
\]
Moreover the convergence is natural in the sense that, given a morphism $A \rightarrow B$, the induced map $L_{p+q}(GF)^I(A) \rightarrow L_{p+q}(GF)^I(B)$ is compatible with the induced map of spectral sequences $(L_pG^I)(L_qF^I)(A) \rightarrow (L_pG^I)(L_qF^I)(B)$.
\end{corollary}
\begin{proof}
We have that each morphism $A^i \rightarrow A^j$ in $C$ induces a morphism
\[
L_{p+q}(GF)(A^i) \rightarrow L_{p+q}(GF)(A^j)
\]
which is compatible with the induced morphisms
\[
(L_pG)(L_qF)(A^i) \rightarrow (L_pG)(L_qF)(A^j).
\]
In other words, giving each component $L_{p+q}(GF)(A^i)$ of $L_{p+q}(GF)^I(A)$ the filtration coming from applying Theorem \ref{gss} to $A^i$ gives a filtration on $L_{p+q}(GF)^I(A)$ whose factors are $(L_pG^I)(L_qF^I)(A)$, as they have $i$th component $(L_pG)(L_qF)(A^i)$. The second part is similar.
\end{proof}

All the results in Section \ref{homdfun} have duals coming from applying the results to opposite categories, since the opposite category of an abelian category is itself abelian. So we call the duals of projectives \emph{injectives}, the duals of left derived functors \emph{right derived functors}, we get that right derived functors are \emph{couniversal} (i.e. satisfying the property dual to being universal), and we get a spectral sequence using injectives instead of projectives.

\section{Homological Bifunctors}
\label{hombifun}

We now consider the case where $C$, $D$ and $E$ are abelian categories, $C$ and $D$ have enough projectives, and $F$ is a right-/right-exact additive bifunctor from $C \times D$ to $E$ (that is, $F$ is right-exact in both variables), covariant in both variables. It is not enough here to fix one variable and take derived functors in the other one: we need long exact sequences in one variable to commute with morphisms in the other. Again, the cases with $F$ contravariant or $C$ or $D$ having enough injectives are similar.

Following the construction of \cite[V.3]{C-E}, we can take left derived functors $F_n=L_nF\colon C \times D \rightarrow E$. The crucial result is \cite[Proposition V.4.1]{C-E}, and we give here a version of it, translated into covariance and left derived functors.

\begin{remark}
The proof in \cite{C-E} is for categories of modules. The general result follows directly from \cite[Freyd-Mitchell Embedding Theorem 1.6.1]{Weibel}.
\end{remark}

\begin{proposition}
\label{connect}
Suppose
\[
\xymatrix{0 \ar[r] & L \ar[r] \ar[d] & M \ar[r] \ar[d] & N \ar[r] \ar[d] & 0 \\
0 \ar[r] & L' \ar[r] & M' \ar[r] & N' \ar[r] & 0}
\]
is a morphism of short exact sequences in $C$, and $A \rightarrow B$ is a morphism in $D$. Then we have a commutative diagram
\small
\[
\xymatrix@C-11pt{\cdots \ar[r] & F_{n+1}(N,A) \ar[r]^{\delta_{n+1}} \ar[d] & F_n(L,A) \ar[r] \ar[d] & F_n(M,A) \ar[r] \ar[d] & F_n(N,A) \ar[r]^{\delta_n} \ar[d] & F_{n-1}(L,A) \ar[r] \ar[d] & \cdots \\
\cdots \ar[r] & F_{n+1}(N',B) \ar[r]^{\delta_{n+1}} & F_n(L',B) \ar[r] & F_n(M',B) \ar[r] & F_n(N',B) \ar[r]^{\delta_n} & F_{n-1}(L',B) \ar[r] & \cdots}
\]
\normalsize
whose rows are exact. Similarly with the variables switched.
\end{proposition}

Note in addition that if, for all $A \in C$ projective, $F(A,-)\colon D \rightarrow E$ is exact, then the $F_n$ are naturally isomorphic to the functors $C \times D \rightarrow E$ given by fixing some $B \in D$ and then taking the derived functors $L_n (F(-,B))$, by \cite[Theorem V.8.1]{C-E}. Similarly with the variables switched. So in this case, we can calculate the derived functors of $F$ just by taking a projective resolution in one variable.

Now, given a small category $I$, we define $F_n^I\colon (C \times D)^I \rightarrow E^I$ to be the exponent of $F_n$ by $I$. We identify $C^I \times D$ with the full subcategory of $(C \times D)^I = C^I \times D^I$ consisting of objects whose $D^I$ component is a `constant' functor, that is, an element of $D^I$ all of whose components are the same, with identity morphisms between them. This gives an additive functor $F_n^I\colon C^I \times D \rightarrow E^I$. Repeating this construction, and identifying $(E^I)^J$ with $E^{I \times J}$, etc., we get a functor $F_n^{I \times J}\colon C^I \times D^J \rightarrow E^{I \times J}$: the main result of this section will be that the $F_n^{I \times J}$, together with the maps $\delta_n^{I \times J}$, satisfy conditions analogous to those of Proposition \ref{connect}. Note that, since $F$ is right-/right-exact, by Lemma \ref{les} $F^I$ is right-/right-exact.

\begin{proposition}
\label{connect'}
Suppose
\[
\xymatrix{0 \ar[r] & L \ar[r] \ar[d] & M \ar[r] \ar[d] & N \ar[r] \ar[d] & 0 \\
0 \ar[r] & L' \ar[r] & M' \ar[r] & N' \ar[r] & 0}
\]
is a morphism of short exact sequences in $C^I$, and suppose
\[
f\colon (\{A^i\},\{\alpha^{ij}\}) \rightarrow (\{B^i\},\{\beta^{ij}\})
\]
is a morphism in $D^J$. Then we have a commutative diagram
\footnotesize
\[
\xymatrix@C-14pt{\cdots \ar[r] & F_{n+1}^{I \times J}(N,A) \ar[r]^{\delta_{n+1}^{I \times J}} \ar[d] & F_n^{I \times J}(L,A) \ar[r] \ar[d] & F_n^{I \times J}(M,A) \ar[r] \ar[d] & F_n^{I \times J}(N,A) \ar[r]^{\delta_n^{I \times J}} \ar[d] & F_{n-1}^{I \times J}(L,A) \ar[r] \ar[d] & \cdots \\
\cdots \ar[r] & F_{n+1}^{I \times J}(N',B) \ar[r]^{\delta_{n+1}^{I \times J}} & F_n^{I \times J}(L',B) \ar[r] & F_n^{I \times J}(M',B) \ar[r] & F_n^{I \times J}(N',B) \ar[r]^{\delta_n^{I \times J}} & F_{n-1}^{I \times J}(L',B) \ar[r] & \cdots}
\]
\normalsize
whose rows are exact. Similarly with the variables switched.
\end{proposition}
\begin{proof}
First we want to show that $\delta_n^{I \times J}$, the map $F_n^{I \times J}(N,A) \rightarrow F_{n-1}^{I \times J}(L,A)$ with components given by the usual map $\delta_n\colon F_n(N^k,A^i) \rightarrow F_{n-1}(L^k,A^i)$ coming from the sequence $0 \rightarrow L^k \rightarrow M^k \rightarrow N^k \rightarrow 0$, really is a morphism in $E^{I \times J}$. For this, we need the squares
\[
\xymatrix{F_n(N^k,A^i) \ar[r]^{\delta_n} \ar[d]_{F_n(N^k,\alpha^{ij})} & F_{n-1}(L^k,A^i) \ar[d]^{F_{n-1}(L^k,\alpha^{ij})} \\
F_n(N^l,A^i) \ar[r]^{\delta_n} & F_{n-1}(L^l,A^i)}
\]
to commute for all $i,j,k,l$. This follows immediately from Proposition \ref{connect}.

Now, by Proposition \ref{F^I}, for each fixed $A^i \in D$, the sequence
\begin{align*}
\cdots &\rightarrow F_{n+1}^I(N,A^i) \xrightarrow{\delta_{n+1}^J} F_n^I(L,A^i) \rightarrow F_n^I(M,A^i) \\
&\rightarrow F_n^I(N,A^i) \xrightarrow{\delta_n^J} F_{n-1}^I(L,A^i) \rightarrow \cdots
\end{align*}
is exact, and similarly for each $B^i$, so by Lemma \ref{les}, each row of our original diagram is exact.

We know that the second and third squares commute by the functoriality of $F_n^{I \times J}$. Finally, to show that
\[
\xymatrix{F_n^{I \times J}(N,A) \ar[r]^{\delta_n^{I \times J}} \ar[d] & F_{n-1}^{I \times J}(L,A) \ar[d] \\
F_n^{I \times J}(N',B) \ar[r]^{\delta_n^{I \times J}} & F_{n-1}^{I \times J}(L',B)}
\]
commutes, we just need
\[
\xymatrix{F_n(N^k,A^i) \ar[r]^{\delta_n} \ar[d] & F_{n-1}(L^k,A^i) \ar[d] \\
F_n(N'^k,B^i) \ar[r]^{\delta_n} & F_{n-1}(L'^k,B^i)}
\]
to commute for all $i,k$. This follows immediately from Proposition \ref{connect}.

The result with the variables switched follows by symmetry, after observing that $(F_n^I)^J=F_n^{I \times J}=(F_n^J)^I$.
\end{proof}

It will also be useful to consider the case where $C$ has enough projectives but $D$ does not, as for example in the case where $C=D$, with enough projectives but not enough injectives (or vice versa), $E=Ab^{op}$, and $F$ is $Hom_C(-,-^{op})^{op}\colon C \times C^{op} \rightarrow Ab^{op}$. In this case, we need an additional hypothesis: that our functor is exact in the second variable whenever we take the first to be projective. Note this is satisfied (by definition of projectivity) in the case of $Hom_C(-,-^{op})^{op}$.

We now assume that $C$, $D$ and $E$ are abelian categories, $C$ has enough projectives, and $F$ is a right-/right-exact additive bifunctor from $C \times D$ to $E$. In particular, for each $A \in D$, we can take the left derived functors of $F(-,A)$ to get a universal homological $\delta$-functor. We will write $F_n(-,A)$ for $L_n(F(-,A))$; the next proposition will show that the $F_n(-,-)$ are in fact bifunctors $C \times D \rightarrow E$.

\begin{proposition}
\label{connect+}
Suppose
\[
\xymatrix{0 \ar[r] & L \ar[r] \ar[d] & M \ar[r] \ar[d] & N \ar[r] \ar[d] & 0 \\
0 \ar[r] & L' \ar[r] & M' \ar[r] & N' \ar[r] & 0}
\]
is a morphism of short exact sequences in $C$, and $A \rightarrow B$ is a morphism in $D$.

Then we have a commutative diagram
\small
\[
\xymatrix@C-11pt{\cdots \ar[r] & F_{n+1}(N,A) \ar[r]^{\delta_{n+1}} \ar[d] & F_n(L,A) \ar[r] \ar[d] & F_n(M,A) \ar[r] \ar[d] & F_n(N,A) \ar[r]^{\delta_n} \ar[d] & F_{n-1}(L,A) \ar[r] \ar[d] & \cdots \\
\cdots \ar[r] & F_{n+1}(N',B) \ar[r]^{\delta_{n+1}} & F_n(L',B) \ar[r] & F_n(M',B) \ar[r] & F_n(N',B) \ar[r]^{\delta_n} & F_{n-1}(L',B) \ar[r] & \cdots}
\]
\normalsize
whose rows are exact, as in Proposition \ref{connect}. Suppose in addition that $F(P,-)$ is exact, for all $P \in C$ projective. Then the proposition holds with the variables switched, i.e. if
\[
\xymatrix{0 \ar[r] & L \ar[r] \ar[d] & M \ar[r] \ar[d] & N \ar[r] \ar[d] & 0 \\
0 \ar[r] & L' \ar[r] & M' \ar[r] & N' \ar[r] & 0}
\]
is a morphism of short exact sequences in $D$, and $f\colon A \rightarrow B$ is a morphism in $C$, then we get a commutative diagram
\small
\[
\xymatrix@C-11pt{\cdots \ar[r] & F_{n+1}(A,N) \ar[r]^{\delta_{n+1}} \ar[d] & F_n(A,L) \ar[r] \ar[d] & F_n(A,M) \ar[r] \ar[d] & F_n(A,N) \ar[r]^{\delta_n} \ar[d] & F_{n-1}(A,L) \ar[r] \ar[d] & \cdots \\
\cdots \ar[r] & F_{n+1}(B,N') \ar[r]^{\delta_{n+1}} & F_n(B,L') \ar[r] & F_n(B,M') \ar[r] & F_n(B,N') \ar[r]^{\delta_n} & F_{n-1}(B,L') \ar[r] & \cdots \rlap{.}}
\]
\normalsize
\end{proposition}
\begin{proof}
In the first case, we can take projective resolutions $L_\ast$ of $L$ and $N_\ast$ of $N$; by the Horseshoe Lemma (\cite[Lemma 2.2.8]{Weibel}), we can construct a projective resolution $M_\ast$ of $M$ whose $n$th term is $L_n \oplus N_n$. So we get a commutative diagram of chain complexes
\[
\xymatrix{0 \ar[r] & F(L_\ast,A) \ar[r] \ar[d] & F(M_\ast,A) \ar[r] \ar[d] & F(N_\ast,A) \ar[r] \ar[d] & 0 \\
0 \ar[r] & F(L'_\ast,B) \ar[r] & F(M'_\ast,B) \ar[r] & F(N'_\ast,B) \ar[r] & 0}
\]
whose rows are exact. Then the result follows by \cite[Proposition 1.3.4]{Weibel}. In the second case, take projective resolutions $A_\ast$ of $A$ and $B_\ast$ of $B$. By the Comparison Theorem (\cite[Theorem 2.2.6]{Weibel}), $f$ induces a unique map of chain complexes $f_\ast\colon A_\ast \rightarrow B_\ast$. Thus we get a commutative diagram of chain complexes
\[
\xymatrix{0 \ar[r] & F(A_\ast,L) \ar[r] \ar[d] & F(A_\ast,M) \ar[r] \ar[d] & F(A_\ast,N) \ar[r] \ar[d] & 0 \\
0 \ar[r] & F(B_\ast,L') \ar[r] & F(B_\ast,M') \ar[r] & F(B_\ast,N') \ar[r] & 0}
\]
whose rows are exact by hypothesis (since $F(P,-)$ is exact for $P$ projective), and the result follows as before.
\end{proof}

\begin{proposition}
\label{connect+'}
Suppose
\[
\xymatrix{0 \ar[r] & L \ar[r] \ar[d] & M \ar[r] \ar[d] & N \ar[r] \ar[d] & 0 \\
0 \ar[r] & L' \ar[r] & M' \ar[r] & N' \ar[r] & 0}
\]
is a morphism of short exact sequences in $C^I$, and $A \rightarrow B$ is a morphism in $D^J$.

Then we have a commutative diagram
\footnotesize
\[
\xymatrix@C-14pt{\cdots \ar[r] & F_{n+1}^{I \times J}(N,A) \ar[r]^{\delta_{n+1}^{I \times J}} \ar[d] & F_n^{I \times J}(L,A) \ar[r] \ar[d] & F_n^{I \times J}(M,A) \ar[r] \ar[d] & F_n^{I \times J}(N,A) \ar[r]^{\delta_n^{I \times J}} \ar[d] & F_{n-1}^{I \times J}(L,A) \ar[r] \ar[d] & \cdots \\
\cdots \ar[r] & F_{n+1}^{I \times J}(N',B) \ar[r]^{\delta_{n+1}^{I \times J}} & F_n^{I \times J}(L',B) \ar[r] & F_n^{I \times J}(M',B) \ar[r] & F_n^{I \times J}(N',B) \ar[r]^{\delta_n^{I \times J}} & F_{n-1}^{I \times J}(L',B) \ar[r] & \cdots}
\]
\normalsize
whose rows are exact, as in Proposition \ref{connect'}. Suppose in addition that $F(P,-)$ is exact, for all $P \in C$ projective. Then the proposition holds with the variables switched, i.e. if
\[
\xymatrix{0 \ar[r] & L \ar[r] \ar[d] & M \ar[r] \ar[d] & N \ar[r] \ar[d] & 0 \\
0 \ar[r] & L' \ar[r] & M' \ar[r] & N' \ar[r] & 0}
\]
is a morphism of short exact sequences in $D^J$, and $A \rightarrow B$ is a morphism in $C^I$, then we get a commutative diagram
\footnotesize
\[
\xymatrix@C-14pt{\cdots \ar[r] & F_{n+1}^{I \times J}(A,N) \ar[r]^{\delta_{n+1}^{I \times J}} \ar[d] & F_n^{I \times J}(A,L) \ar[r] \ar[d] & F_n^{I \times J}(A,M) \ar[r] \ar[d] & F_n^{I \times J}(A,N) \ar[r]^{\delta_n^{I \times J}} \ar[d] & F_{n-1}^{I \times J}(A,L) \ar[r] \ar[d] & \cdots \\
\cdots \ar[r] & F_{n+1}^{I \times J}(B,N') \ar[r]^{\delta_{n+1}^{I \times J}} & F_n^{I \times J}(B,L') \ar[r] & F_n^{I \times J}(B,M') \ar[r] & F_n^{I \times J}(B,N') \ar[r]^{\delta_n^{I \times J}} & F_{n-1}^{I \times J}(B,L') \ar[r] & \cdots \rlap{.}}
\]
\normalsize
\end{proposition}
\begin{proof}
This follows from Proposition \ref{connect+} in the same way that Proposition \ref{connect} follows from Proposition \ref{connect'}; one just needs to check the commutativity of certain squares, which are immediate consequences of it.
\end{proof}


\begin{thebibliography}{9}

\bibitem{C-E}
Cartan, H, Eilenberg, S: \emph{Homological algebra}, Princeton University Press, Princeton, 1956.

\bibitem{Myself2}
Corob Cook, G: Bieri-Eckmann criteria for profinite groups, preprint (2014).

\bibitem{risingc}
Murfet, D: Abelian categories, \emph{The Rising Sea} (2006), available at \url{http://therisingsea.org/notes/AbelianCategories.pdf}. 

\bibitem{Weibel}
Weibel, C: \emph{An introduction to homological algebra} (Cambridge Studies in Advanced Mathematics, 38), Cambridge University Press, Cambridge, 1994.

\end{thebibliography}
\end{document}